\documentclass[12pt]{amsart}
\usepackage{amssymb, amsmath, amsmath}
\usepackage[backref]{hyperref}
\usepackage{mathrsfs}
\usepackage[alphabetic,backrefs,lite]{amsrefs}
\usepackage{amscd}   
\usepackage{fullpage}
\usepackage[all]{xy} 
\usepackage{setspace}

\DeclareFontEncoding{OT2}{}{} 

\usepackage[usenames,dvipsnames]{color}

\newtheorem{lemma}{Lemma}[section]
\newtheorem{theorem}[lemma]{Theorem}
\newtheorem{proposition}[lemma]{Proposition}

\newtheorem{claim*}{Claim}

\theoremstyle{definition}
\newtheorem{remark}[lemma]{Remark}

\newcommand{\PP}{{\mathbb P}}

\newcommand{\F}{{\mathbb F}}

\newcommand{\calS}{{\mathcal S}}

\def\PP{\mathbb{P}}

\author[Aubry]{Yves Aubry}
\address[Aubry]{Institut de Math\'ematiques de Toulon - IMATH, Universit\'e de Toulon, France}
\email{yves.aubry@univ-tln.fr}

\address[Aubry]{Institut de Math\'ematiques de Marseille - I2M, Aix Marseille Univ, UMR 7373 CNRS, France}
\email{yves.aubry@univ-amu.fr}

\author[Voloch]{Jos\'e Felipe Voloch}
\address[Voloch]{School of Mathematics and Statistics, University of Canterbury, Christchurch 8140, New Zealand}
\email{felipe.voloch@canterbury.ac.nz}
\urladdr{https://www.math.canterbury.ac.nz/~f.voloch/}

\title{Rational points on smooth surfaces in $\PP^3$ over finite fields}

\begin{document}

\begin{abstract}
We improve a bound due to the second author on number of rational points on smooth surfaces in $\PP^3$ over finite fields and look at families of surfaces that achieve or nearly achieve this bound, for which we compute their exact number of rational points. These computations may have independent interest.
\end{abstract}

\maketitle



\section{Introduction}

Let $S$ be  a smooth surface in $\PP^3$ of degree $d$ defined over the finite field ${\mathbb F}_p$ with $p$ prime.
We are interested in the number $\sharp S({\mathbb F}_p)$ of rational points of $S$ over ${\mathbb F}_p$.
Deligne has proved, as a consequence of his proof of the Weil conjectures in \cite{Deligne}, that
$$\sharp S({\mathbb F}_p) \leq p^2+1 +(d^3-4d^2+6d-2)p.$$

There are, in addition, several bounds that can be obtained by elementary methods. See \cites{Homma-Kim-1,Homma-Kim-2} and the references therein. An example of such bound was proved by the second author in Proposition 2 of \cite{Voloch}, namely
$$\sharp S({\mathbb F}_p) \leq (d-1)(p+1)^2+p+1$$
which is better than Deligne's bound if $d>\sqrt{p/3}$ approximately.
Moreover, if we suppose that $S$ does not contain a line defined over ${\mathbb F}_p$, then the above bound can be improved to
$$\sharp S({\mathbb F}_p) \leq (d-1)p^2+(d-2)(p+1)+1.$$

By a different method, inspired by \cite{SV}, the second author proved in Theorem 2 of \cite{Voloch} that if $2<d<p$ and if $m$ denotes the number of ${\mathbb F}_p$-lines contained in 
$S$, then we have
\begin{equation}\label{Felipe}
\sharp S({\mathbb F}_p) \leq \frac{d(d+p-1)(d+2p-2)}{6}+m(p+1)
\end{equation}
and so, in particular:
\begin{equation}\label{Felipe_plus}
\sharp S({\mathbb F}_p) \leq \frac{d(d+p-1)(d+2p-2)}{6}+d(11d-24)(p+1).
\end{equation}
The main purpose of this paper is to improve this last bound and to provide examples of surfaces for which we are able to determine its number of rational points together the number of ${\mathbb F}_p$-lines contained in it and that achieve or get close to this bound.

\subsection*{Acknowledgements}
The first author would like to thank the GAATI laboratory at the Universit\'e de Polyn\'esie Fran\c caise for his hospitality during the academic year 2025/2026 in a CNRS position and is grateful to the Canterbury University at Christchurch in New Zealand where this work was initiated.

The second author was supported by the Marsden Fund administered by the Royal
Society of New Zealand. He would like to thank H. Borges, E. Esteves and S. Kleiman for helpful conversations.


\section{The main result}

Let $S$ be a smooth surface in $\PP^3$ of degree $d$ defined over the finite field ${\mathbb F}_p$ with $p$ prime and with $2<d<p$
and let $L_1,\ldots, L_m$ be the ${\mathbb F}_p$-lines on $S$.

Suppose that the surface $S$ is given as the set of zeros in ${\mathbb P}^3$ of a homogeneous polynomial $f(x_0,x_1,x_2,x_3)$ with coefficients in ${\mathbb F}_p$. Then consider the surface $S_1$ given as the zero locus of the polynomial $\sum_{i=0}^3(\partial f/\partial x_i)x_i^p$ and the surface $S_2$ given by the polynomial
$\sum_{i=0}^3\sum_{j=0}^3(\partial^2 f/\partial x_ix_j)x_i^px_j^p$. 

We consider, as in \cite{Voloch}, the scheme $X:=S\cap S_1\cap S_2$.
The following lemma will allow us to estimate the zero-dimensional part of $X$.

\begin{lemma}\label{Fondamental_Lemma}
    Let $\calS$ be a smooth surface in $\PP^3$ of degree $d$ and $\calS_i$ surfaces in $\PP^3$ of degree $d_i, i=1,2$. Assume that $\calS,\calS_1,\calS_2$ intersect in a zero-dimensional scheme $\Gamma$ and lines $L_1,\ldots,L_m$ with multiplicity one. Then
    $$2m \le \deg \Gamma - (dd_1d_2 -m(d+d_1+d_2)) \le m(m+1).$$
\end{lemma}

\begin{proof}
    We work with intersection products on $\calS$. The surfaces $\calS_i$ cut $\calS$ in divisors $C_i +L_1+\cdots+L_m$ and we are interested in $\Gamma = C_1\cdot C_2$. So, letting $H$ denotes a hyperplane section of $\calS$, we have
    $$\deg \Gamma = (d_1H - (L_1+\cdots+L_m))\cdot (d_2H - (L_1+\cdots+L_m)).$$

    By classical intersection theory on surfaces, one have the following: 
    $$H^2 = d,\ \  L_i.H = 1, \ \ {\rm and}\ \   L_i^2 = 2-d.$$
    The first follows since $\calS$ has degree $d$.
    For the second, a line and a plane meet at one point but if the line is contained in $\calS$, the intersection point is also in $\calS$. The third follows the previous one and the adjunction formula. 
 Indeed, if $L$ is a nonsingular curve of genus $g$ on a surface $S$, and $K$ is the canonical divisor on $S$, then the adjunction formula gives:
 $$2g-2=L.(L+K)=L^2+L.K.$$
 If $L=L_i$ is a line, its genus is zero and we get
 $$-2=L_i^2+L_i.K.$$
 Moreover, the canonical divisor $K$ on a smooth surface in $\PP^3$ of degree $d$ verifies:
 $$K\sim (d-4)H$$
 where $H$ is a hyperplane section. So we obtain
 $-2=L_i^2+(d-4)L_i.H$ and thus $L_i^2=2-d$ since $L_i.H=1$.
 
    Hence, the right hand side of the last equation expands to 

    $$d_1d_2H^2 - (d_1+d_2)(L_1+\cdots+L_m)\cdot H + (L_1+\cdots+L_m)^2 = dd_1d_2 -m(d_1+d_2) + (L_1+\cdots+L_m)^2.$$

    Finally, we have that $L_i\cdot L_j = 0$ or $1$ if $i\ne j$, depending on whether the lines intersect or not, so

    $$(L_1+\cdots+L_m)^2 = \sum_{i=1}^m L_i^2 + \sum_{i \ne j} L_i\cdot L_j \le m(2-d)+m(m-1)=m(m+1-d)$$
and
    $$(L_1+\cdots+L_m)^2 = \sum_{i=1}^m L_i^2 + \sum_{i \ne j} L_i\cdot L_j \ge m(2-d)$$
    
    which combine with the above to prove the lemma.
\end{proof}

The next theorem improves Bound (\ref{Felipe}).

\begin{theorem}\label{Main_Theorem}
Let $S$ be a smooth surface in $\PP^3$ of degree $d$ defined over the finite field ${\mathbb F}_p$ with $p$ prime and $2<d<p$. 
Let  $m$ be the number of lines defined over ${\mathbb F}_p$ on $S$ and suppose that they have multiplicity one in $X:=S\cap S_1\cap S_2$,
where, if $S$ is given by the zero locus of the polynomial
$f(x_0,x_1,x_2,x_3)$ then
 $S_1$ and $S_2$ are the surfaces given respectively by the zero locus of 
the polynomial $\sum_{i=0}^3(\partial f/\partial x_i)x_i^p$ and
the polynomial
$\sum_{i=0}^3\sum_{j=0}^3(\partial^2 f/\partial x_ix_j)x_i^px_j^p$.

Then the number of rational points on $S$ is upper bounded by:
\begin{equation}\label{Improved_bound}
\sharp S({\mathbb F}_p) \leq \frac{d(d+p-1)(d+2p-2)}{6}
-\frac{3m(d+p-1)-m(m+1)}{6}
+m(p+1).
\end{equation}

In particular we have:
\begin{equation}\label{Improved_without_m}
\sharp S({\mathbb F}_p) \leq \frac{d(d+p-1)(d+2p-2)}{6}
+\frac{(11d^2-30d+18)(11d^2-33d +28 +3p)}{6}.
\end{equation}

\end{theorem}

\begin{proof}
In \cite{Voloch} it is proved that the one dimensional components of $X$ are lines. Hence, the situation here is as in Lemma \ref{Fondamental_Lemma}, with the surface $S_1$ of degree $d_1 = d + p -1$, and the surface $S_2$ of degree $d_2 =d+2p -2$.
So, since we have supposed that the ${\mathbb F}_p$-lines $L_1,\ldots, L_m$ on $S$ have multiplicity one in $X$, we have by Lemma \ref{Fondamental_Lemma} that the zero-dimensional subscheme $\Gamma$ of $X$ given by its isolated rational points verifies:

$$\deg \Gamma \le d(d + p -1)(d+2p -2) -m(3d +3p-3) + m(m+1).$$

Moreover, is has been proved in the proof of Theorem 2 of  \cite{Voloch} that, outside of the lines contained  in $S$,  
the rational points of $S$ are isolated points of $S\cap S_1\cap S_2$ with multiplicity at least $6$.
So, the rational points of $S$ not on lines have multiplicity at least $6$ in $\Gamma$ and thus
 we get the following bound:

$$\# S(\F_p) \le \frac{d(d + p -1)(d+2p -2) -m(3d +3p-3) + m(m+1)}{6} + m(p+1)$$

which can be rewritten as

$$\# S(\F_p) \le \frac{d(d + p -1)(d+2p -2) -m(3d-7-m)}{6} + \frac{m(p+1)}{2}.$$

Furthermore,  
Bauer and Rams have proved in Theorem 1 in \cite{Bauer-Rams}
 that a smooth surface in ${\mathbb P}^3$  of degree $d>2$ over a field of characteristic 0 or of characteristic $p>d$  contains at most
 $11d^2-30d+18$  lines. So the last bound of the theorem follows.
\end{proof}

\begin{remark}
In order to see  when Theorem \ref{Main_Theorem}  improves Bound (\ref{Felipe}), 
we set 
$$G:=\frac{3m(d+p-1)-m(m+1)}{6}.$$
It is easy to see that $G\geq 0$ if and only if $m\leq 3d+3p-4$.
But we have already seen that, as a smooth surface in ${\mathbb P}^3$ of degree $2<d<p$ defined over ${\mathbb F}_p$, $S$ contains at most $11d^2-30d+18$ lines. So, as soon as $11d^2-30d+18\leq 3d+3p-4$, we will have $G\geq 0$. Thus, if $p> \frac{11d^2-33d+22}{3}$ then $G> 0$ and Theorem \ref{Main_Theorem}  improves Bound (\ref{Felipe}).
\end{remark}

\begin{remark}
We have tried but not succeeded in proving that, under the other hypothesis of Theorem \ref{Main_Theorem}, that the lines of $X:=S\cap S_1\cap S_2$ always occur with multiplicity one. We will prove that this is the case for Fermat surfaces in Proposition \ref{Number_Lines}. We have also verified by computer that this is so on a number of examples. We expect that this is always the case.

\end{remark}

\begin{remark}
Let $S$ be a smooth surface in $\PP^3$ of degree $d$ defined over a finite field ${\mathbb F}_q$ of characteristic $p$.

If $d=3$ then $S$ is a nonsingular cubic surface and it is well known that $S$ contains at most 27 lines. But one knows very precise results on its number of rational points. Indeed one has
$$\sharp S({\mathbb F}_q) = q^2+nq+1$$
where $-2\leq n\leq 7$ with $n\not=6$  (see for example \cite{Manin}, Table 1).
 Moreover, Swinnerton-Dyer proved in \cite{SW} that 
for $q=2,3$ or $5$ then we have $n\leq 5$. Following the proof of Theorem \ref{Main_Theorem}, all points in $X:=S\cap S_1\cap S_2$ that are not on a line are rational. Indeed, an asymptotic line at a point of $S$ that is not contained in $S$, does not meet $S$ again, since $S$ has degree $3$. These points occur with multiplicity $6$ on $X$ if, furthermore, they are not contained in an (irrational) line of $S$. We hoped to recover the known results on the number of rational points of cubic surfaces from these facts but have not succeeded.

If $d=4$ and $p\not=2,3$ then Rams and Sch\"utt in \cite{R-S} proved that a smooth quartic surface $S$ contains at most 64 lines and this bound is sharp.
So for $p\geq 5$
Bound(\ref{Felipe}) gives
$$\sharp S({\mathbb F}_p) \leq \frac{4(p+3)(2p+2)}{6}
+64p+64,$$
which is improved by Bound (\ref{Improved_bound}) of Theorem \ref{Main_Theorem}
as soon as $p\geq 19$ in the following way:
$$\sharp S({\mathbb F}_p) \leq \frac{4(p+3)(2p+2)}{6}
+32p+661.$$

\end{remark}



\section{Surfaces without isolated rational points}

Consider Fermat type surfaces $S$ in ${\mathbb P}^3$ defined over ${\mathbb F}_p$ with $p\equiv 1\pmod d$ given by the equation:
$$x^d + y^d -z^d -w^d=0.$$

\begin{proposition}\label{Number_Lines}
If $-1$ is a $d$-th power modulo $p$ then the  surface $S$ contains $3d^2$ lines defined over ${\mathbb F}_p$ and
all the lines in the intersection $X:=S\cap S_1 \cap S_2$ have multiplicity one.
\end{proposition}

\begin{proof}
Following \cite{ARM} (which ostensibly is over the complex numbers field $\mathbb C$, but one can see that the proofs work over finite fields) one can prove that  $S$ has $3d^2$ lines defined over ${\mathbb F}_p$, namely the lines with equations
\begin{equation*}
\left\{
\begin{matrix}
w = \eta^i x   \\
y = \eta^kz \cr
\end{matrix}
\right.
\qquad
\left\{
\begin{matrix}
x = \eta^{k+i} z   \\
w = \eta^iy\cr
\end{matrix}
\right. 
\qquad
\left\{
\begin{matrix}
x = v\eta^i y   \\
w = \eta^{k+i}z \cr
\end{matrix}
\right. 
\end{equation*}
where $\eta$ is a primitive $d$-th root of unity in ${\mathbb F}_p$ and $v\in {\mathbb F}_p$ such that $v^d=-1$.

Moreover, one can prove that already for $S\cap S_1$ (and a posteriori for $X$) the multiplicity of the lines is one. We show this by verifying that $S$ and $S_1$ intersect transversely outside of a finite set. We work in the affine patch $w=1$ (and the other patches are similar). Then $S$ has equation $x^d +y^d - z^d =1$ with gradient $(dx^{d-1}, dy^{d-1},-dz^{d-1})$ and 
$S_1$ has equation $x^{d+q-1} +y^{d+q-1} - z^{d+q-1} =1$ with gradient $((d-1)x^{d+q-2}, (d-1)y^{d+q-2},-(d-1)z^{d+q-2})$.

The intersection $S\cap S_1$ is not transversal at points where the gradients are proportional. For this to happen, there is a constant $\lambda$ for which $dx^{d-1} = \lambda (d-1)x^{d+q-2}$ and similarly for $y,z$. So, $x^{q-1} = d/(\lambda (d-1))$ and similarly for $y,z$. So, $x= \lambda^{-1/(q-1)}c_x, y= \lambda^{-1/(q-1)}c_y, z= \lambda^{-1/(q-1)}c_z$ where $c_x,c_y,c_z$ come from the same finite set $c^{q-1} = d/(d-1)$. But, using that $(x,y,z) \in S$,  we get 
$\lambda^{d/(q-1)} = c_x^d+c_y^d+c_z^d$, so $\lambda$ is also on a finite set, proving that $S$ and $S_1$ intersect transversely outside of a finite set.
\end{proof}



Consider now the surface $S$ in ${\mathbb P}^3$ defined over ${\mathbb F}_p$ given by the equation:
$$x^{(p-1)/2} + y^{(p-1)/2} -z^{(p-1)/2} -w^{(p-1)/2}=0.  $$

If we denote by $f$ the polynomial $x^{(p-1)/2} + y^{(p-1)/2} -z^{(p-1)/2} -w^{(p-1)/2}$ then
the two surfaces $S_1$ and $S_2$ given respectively by the polynomials
 $\sum_{i=0}^3(\partial f/\partial x_i)x_i^p$ and 
$\sum_{i=0}^3\sum_{j=0}^3(\partial^2 f/\partial x_ix_j)x_i^px_j^p$ have equation respectively:

$$x^{3(p-1)/2} + y^{3(p-1)/2} -z^{3(p-1)/2} -w^{3(p-1)/2}=0$$
and
$$x^{5(p-1)/2} + y^{5(p-1)/2} -z^{5(p-1)/2} -w^{5(p-1)/2}=0.$$

We give below the number of rational points on $S$ and we show that the scheme $X:=S\cap S_1\cap S_2$ does not contain isolated rational points.
\begin{proposition}
The number of rational points of the surface $S$ in ${\mathbb P}^3$ defined over ${\mathbb F}_p$  by 
$x^{(p-1)/2} + y^{(p-1)/2} -z^{(p-1)/2} -w^{(p-1)/2}=0$ is given by:

$$\sharp S({\mathbb F}_p)=\frac{3}{8}(p^3-3p^2+11p-9).$$
This is an example of a surface such that $S\cap S_1\cap S_2$ does not contain isolated rational points.
\end{proposition}

\begin{proof}
Let us remark that if $a\in{\mathbb F}_p$ then, $a^{(p-1)/2}=0$ if and only if $a=0$ and  it is equal to 1 if $a$ is a quadratic residue modulo $p$ and is equal to $-1$ otherwise. And it is well known that the number of nonzero quadratic residues, as the number of  quadratic nonresidues modulo $p$ is equal to $(p-1)/2$. We begin by counting the number of affine rational points on $S$. When none of the coordinates are zero, we have 6 choices for the positions of the $+1$, so $6(\frac{p-1}{2})^4$ solutions. If two coordinates are zero, then we have also 6 choices for the zero coordinates and then two possibilities for the choices of $+1$ or $-1$, thus $6\times 2 \times (\frac{p-1}{2})^2$ solutions. Adding the point with all coordinate zero, we obtain the number $N_{\rm aff}$ of affine rational points on $S$ which is equal to 
$6(\frac{p-1}{2})^4+12 (\frac{p-1}{2})^2+1$ and then $\sharp S({\mathbb F}_p) =\frac{N_{\rm aff}-1}{p-1}$ which gives the result.

Moreover, let us calculate the degree of $\Gamma$, the zero-dimensional scheme appearing in the intersection 
$X:=S\cap S_1\cap S_2$.
By Proposition \ref{Number_Lines} we have that the lines  
in $X$ have multiplicity one, and so we can apply
Lemma \ref{Fondamental_Lemma} and we get
$\deg\Gamma=dd_1d_2 - m(d+d_1+d_2)+2m+ \sum_{i \ne j} L_i\cdot L_j$
with, here, $d=(p-1)/2, d_1=3(p-1)/2$ and $d_2=5(p-1)/2$. Let us scrutinize now the intersections of the lines $L_i$.
Since $-1$ is clearly a $(p-1)/2$-th power modulo $p$, one can use
 Proposition \ref{Number_Lines} to asserts that the lines $L_i$ contained in  $X$
 are of three types:
\begin{equation*}
\left\{
\begin{matrix}
x = ay   \\
z = bw \cr
\end{matrix}
\right.
\qquad
\left\{
\begin{matrix}
x = a'z   \\
y = b'w\cr
\end{matrix}
\right. 
\qquad
\left\{
\begin{matrix}
x = a' w   \\
y = b'z \cr
\end{matrix}
\right. 
\end{equation*}
where $a,b$ are quadratic non-residues modulo $p$ and $a',b'$ are (nonzero) quadratic residues modulo $p$.
Thus we obtain, according to Proposition \ref{Number_Lines}, that  the number of lines is $m=3(\frac{p-1}{2})^2$.

If $L_1$ and $L_2$ are two lines of same type given, for example, by equations
\begin{equation*}
\left\{
\begin{matrix}
x = a_1y   \\
z = b_1w \cr
\end{matrix}
\right.
\qquad
\left\{
\begin{matrix}
x = a_2y   \\
z = b_2w\cr
\end{matrix}
\right. 
\end{equation*}
then $L_1.L_2=1$ if and only if $a_1=a_2$ or $b_1=b_2$.
Hence the contribution will be in the case of lines $L_i$ and $L_j$ of same type:
$$\sum_{i\not=j, {\rm same\ type}}L_i.L_j=3\left(\frac{p-1}{2}\right)^2\left((\frac{p-1}{2}-1)\times 2\right)=3\left(\frac{p-1}{2}\right)^2(p-3).$$

Moreover, if $L_1$ and $L_2$ are two lines of different type given, for example, by equations
\begin{equation*}
\left\{
\begin{matrix}
x = ay   \\
z = bw \cr
\end{matrix}
\right.
\qquad
\left\{
\begin{matrix}
x = a'z   \\
y = b'w\cr
\end{matrix}
\right. 
\end{equation*}
then $L_1.L_2=1$ if and only if $ab'=a'b$.
Hence the contribution will be in the case of lines $L_i$ and $L_j$ of different type:
$$\sum_{i\not=j, {\rm different\  type}}L_i.L_j=6\left(\frac{p-1}{2}\right)^3.$$
Finally, we obtain:
$$\sum_{i\not=j}L_i.L_j=6\left(\frac{p-1}{2}\right)^2(p-2).$$

One can compute now 
$\deg\Gamma= dd_1d_2 - m(d+d_1+d_2)+2m+ \sum_{i \ne j} L_i\cdot L_j$ which gives
$$\deg\Gamma= 15\left(\frac{p-1}{2}\right)^3-3\left(\frac{p-1}{2}\right)^2\left(9\left(\frac{p-1}{2}\right)\right)+6\left(\frac{p-1}{2}\right)^2+6\left(\frac{p-1}{2}\right)^2(p-2)=0.$$
Thus $S\cap S_1\cap S_2$ does not contain isolated rational points.
\end{proof}



\section{Surfaces without rational lines}
\label{nolines}

We are interested in the surface
 $S$ in ${\mathbb P}^3$ defined over ${\mathbb F}_p$  by the equation:
$$x^{2(p-1)/5} + y^{2(p-1)/5} + x^{(p-1)/5} y^{(p-1)/5}+z^{2(p-1)/5}  + w^{2(p-1)/5}=0$$
with $p\equiv 1\pmod 5$.

Let us remark that if $t\in{\mathbb F}_p^{\ast}$, then $(t^{2(p-1)/5})^5=1$ and thus $t^{2(p-1)/5}$ is a 5-th root of unity in ${\mathbb F}_p$.
It is well known that, if $\zeta$ is a 5-th primitive root of unity, then $1+\zeta+\zeta^2+\zeta^3+\zeta^4=0$.
Except for certain values of $p$, one can show that 
this is the only way that
a sum of at most five 5-th roots of unity in ${\mathbb F}_p$ can be zero,
i.e.
a sum of at most five 5-th roots of unity in ${\mathbb F}_p$ can be zero
 only if they are distinct:

\begin{lemma}\label{sum_roots}
Let $p$ be a prime such that $p\equiv 1\pmod 5$ and $p\not=11, 41, 61$.
Then, a sum of at most five 5-th roots of unity in ${\mathbb F}_p$ can be zero only if 
this is the sum of the five distinct 5-th roots of unity.
\end{lemma}

\begin{proof}
Let $\zeta_5$ be a 5-th primitive root of unity in $\mathbb C$ and $\mu_5({\mathbb C})$ the group of 5-th roots of unity in ${\mathbb C}$.
We consider in the cyclotomic number field ${\mathbb Q}(\zeta_5)$  the sums
$x_1+x_2+x_3+x_4+x_5$ where $x_i\in\mu_5({\mathbb C})\cup \{0\}$. We are looking for the primes $p\equiv 1\pmod 5$ dividing the norm of such sums.

Consider for example the algebraic integer $\alpha:=1+1+\zeta_5+\zeta_5+\zeta_5\in {\mathbb Q}(\zeta_5)$. The norm of $\alpha$ is equal to 
$$N_{{\mathbb Q}(\zeta_5)/{\mathbb Q}}(\alpha)=\prod_{(k,5)=1}\sigma_k(\alpha)$$
where $\sigma_k(\zeta_5)=\zeta_5^k$. Hence $N_{{\mathbb Q}(\zeta_5)/{\mathbb Q}}(\alpha)=3^4\prod_{k=1}^4\left(-\frac{2}{3}-\zeta_5^k\right)=3^4\Phi_5(-\frac{2}{3})$ where $\Phi_5(X)=\prod_{k=1}^4(X-\zeta_5^k)=X^4+X^3+X^2+X+1$ is the 5-th cyclotomic polynomial. Thus $N_{{\mathbb Q}(\zeta_5)/{\mathbb Q}}(\alpha)=5\times 11$.

Since $11\equiv 1\pmod 5$, the prime $p=11$ totally splits in the extension ${\mathbb Q}(\zeta_5)/{\mathbb Q}$ and the quotient of the ring of integers 
${\mathbb Z}[\zeta_5]$ by the ideal generated by $11$ is a product of fields
$${\mathbb Z}[\zeta_5]/11{\mathbb Z}[\zeta_5]\simeq {\mathbb F}_{11}^4.$$
Then we look at the image of the algebraic integer $2+3\zeta_5$ by the reduction morphisms which sends $\zeta_5$ to one of the 5-th primitive root of unity in ${\mathbb F}_{11}$ which are $\{3,9,5,4\}$. We find that $2+3\times 3\equiv 0\pmod{11}$ i.e. $1+1+3+3+3$ is a sum of five non-distincts 5-th root of unity in ${\mathbb F}_{11}$ which is equal to zero.

In the same way, we show that the algebraic integer $1+4\zeta_5\in{\mathbb Z}[\zeta_5]$ has norm equal to $5\times 41$, and the reduction morphism which sends $\zeta_5$ to  $10$, which is a 5-th primitive root of unity in ${\mathbb F}_{41}$, gives $1+10+10+10+10=0$,  a sum of five non-distincts 5-th root of unity in ${\mathbb F}_{41}$ which is equal to zero.

Again, the algebraic integers $1+3\zeta_5$ and $3+\zeta_5$ have norms equal to 61 and provide the following sums of four 5-th roots of unity in ${\mathbb F}_{61}$ which is zero: 
$20+20+20+1=0$ and $1+1+1+58=0$.

Finally, if $\alpha=x_1+x_2+x_3+x_4+x_5$ where $x_i\in\mu_5({\mathbb C})\cup \{0\}$, then $N_{{\mathbb Q}(\zeta_5)/{\mathbb Q}}(\alpha)$ is a product of at most five complex numbers of modulus at most 5, and thus the norm of $\alpha$ is upper bounded by $5^4=625$.
One can show that, aside from $11$, $41$ and $61$, no other prime number $\equiv 1\pmod 5$ and less than $625$ appears as a factor of the norm of such an element $\alpha$.
Thus, if $p\not= 11, 41, 61$ then the only way to have a  sum of at most five 5-th roots of unity in ${\mathbb F}_p$ to be zero, up to permutations, is 
$1+\zeta+\zeta^2+\zeta^3+\zeta^4=0$ where $\zeta$ is a 5-th primitive root of unity in  ${\mathbb F}_p$.
\end{proof}

\begin{proposition}
The surface $S$ in  ${\mathbb P}^3$ defined over ${\mathbb F}_p$ given by the equation:
$$x^{2(p-1)/5} + y^{2(p-1)/5} + x^{(p-1)/5} y^{(p-1)/5}+z^{2(p-1)/5} + w^{2(p-1)/5}=0$$
with $p\equiv 1\pmod 5$,  does not contain lines defined over ${\mathbb F}_p$.

If $p\not=11, 41, 61$  then we have:
$$\sharp S({\mathbb F}_p)=8\left(\frac{p-1}{5}\right)^3.$$
Moreover, if $p=11$ we have $\sharp S({\mathbb F}_{p})={18} \left(\frac{p-1}{5}\right)^3={144}$,
 if $p=41$ we have  $\sharp S({\mathbb F}_{p})=10 \left(\frac{p-1}{5}\right)^3=5120$
 {and if $p=61$ we have $\sharp S({\mathbb F}_{p})=8 \left(\frac{p-1}{5}\right)^3+2 \left(\frac{p-1}{5}\right)^2=
 14 112$.}
\end{proposition}

\begin{proof}
Let us prove the first assertion.
Let $H$ be the hyperplane given by the equation $x=0$. The intersection $S\cap H$ does not contain any rational point since it would provides
a sum of at most three 5-th root of unity equal to zero and this cannot happens by Lemma \ref{sum_roots}.
We deduce that $S$ does not contain a line defined over ${\mathbb F}_p$ (because such a line would intersects the plane $H$ at a rational point).

We have now to count the number of points $(x:y:z:w)\in {\mathbb P}^3({\mathbb F}_p)$ such that 
$x^{2(p-1)/5} + y^{2(p-1)/5} + x^{(p-1)/5} y^{(p-1)/5}+z^{2(p-1)/5} + w^{2(p-1)/5}=0$.
If $p\not=11,41, 61$, by Lemma \ref{sum_roots}, this is necessarily a sum of five distincts 5-th roots of unity.
Consider the morphism
$$
\begin{matrix}
\psi:&{\mathbb F}_p^{\ast}\longrightarrow &{\mathbb F}_p^{\ast}\cr
&x\longmapsto  & x^{2(p-1)/5}.\cr
\end{matrix}
$$
The kernel has order $\frac{p-1}{5}$ and the image is the group of 5-th roots of unity $\mu_5({\mathbb F}_p)$ in ${\mathbb F}_p^{\ast}$.
So for the choice for $x$, any  element of ${\mathbb F}_p^{\ast}$ works, so we have $p-1$ choices.
Then, for $y$ one can choose $p-1-\frac{p-1}{5}=\frac{4(p-1)}{5}$ differents element, and finally $x^{(p-1)/5} y^{(p-1)/5}$ is completely
determined. It remains $\frac{2(p-1)}{5}$ choices for $z$, and then $\frac{p-1}{5}$ choices for $w$.
So, dividing $8(p-1)\left(\frac{p-1}{5}\right)^3$ by $p-1$ to find projective points, we get the result.

For $p=11, 41$ and $61$, a direct computation gives the result. Extra points come from sums of less than five 5-th roots of unity which are equal to zero or from sums of five nondistinct 5th roots of unity which are also equal to zero.

For example, we have $1+1+3+3+3=1+5+5=1+1+1+4+4=1+1+9=1+9+4+4+4=1+3+9+9=1+1+4+5=3+3+5=1+3+3+4=1+5+9+9+9=0$ in ${\mathbb F}_{11}$.
\end{proof}

\begin{remark}
Consider the surfaces in ${\mathbb P}^3$ given by the following equation
$$x^{2d}+x^dy^d+y^{2d}+z^{2d}+w^{2d}=0$$ 
and defined over a field 
 where they have lines 
 \begin{equation}\label{lines}
\left\{
\begin{matrix}
y = \omega x   \\
z = \lambda w \cr
\end{matrix}
\right.
\qquad
{\rm with}\  \omega^{2d}+\omega^d + 1=0\ {\rm and}\ \lambda^{2d}+1=0.
\end{equation}
One can show that all the lines in the intersection $X:=S\cap S_1 \cap S_2$ have multiplicity one.

Note that if $d=\frac{p-1}{5}$, we find again the surfaces defined over ${\mathbb F}_p$ with $p\equiv 1\pmod 5$ of the previous proposition but
the lines given in (\ref{lines}) are not defined over ${\mathbb F}_p$.

Remark also that the result about the multiplicity of the lines in $S\cap S_1 \cap S_2$ is false in general since it is false for the surface  $zw=x^2$ and the line $y=z=0$ (but this surface is singular).
\end{remark}

\begin{remark}
The bound (\ref{Felipe}) gives $\sharp S({\mathbb F}_p)\leq 28\left(\frac{p-1}{5}\right)^3.$
\end{remark}



\section{Other examples}

In this section, we collect a few other examples for which our techniques can be applied to compute a lower bound on the number of their rational points.

Define
$$
G\left(u_0, u_1, u_2, u_3\right)= u_0^2 + u_0 u_1+u_1^2+u_0 u_2 +u_2^2+u_2 u_3+u_3^2 
$$
and
$$
F\left(x, y, z, w\right)=G(x^{\frac{p-1}{7}}, y^{\frac{p-1}{7}}, z^{\frac{p-1}{7}}, w^{\frac{p-1}{7}}).
$$

If $\zeta$ is a 7-th primitive root of unity then:
$$
G\left(\zeta, \zeta^3, \zeta^{-1}, \zeta^{-3}\right)= \zeta^2+\zeta^4+\zeta^6+1 +\zeta^{-2}+\zeta^{-4}+\zeta^{-6}=0
$$
and
$$
G\left(\zeta^3, \zeta, \zeta^{-3}, \zeta^{-1}\right)=0.
$$

We checked numerically that there are no points on the surface when $u_2=0,u_3=1$ for $p<300,p \ne 29$ so there are no lines on the surface either for those values of $p$. When $p=29$, there are points with $u_2=0,u_3=1$ from $N_{{\mathbb Q}(\zeta_7)/{\mathbb Q}}(\zeta_7^2+\zeta_7+2)=29$.

The surface $F=0$ has at least $12((p-1)/7)^3$ points coming from the above identities. 
For $p=29,43,71,239,337$ it has more points, at least $m((p-1)/7)^3$ for $m=18,17,16,14,14$ respectively.

\bigskip

Finally, define
$$
G_2(u_0, u_1, u_2, u_3)=\sum_{i\le j} u_iu_j
$$
and
$$
F_2\left(x, y, z, w\right)=G_2(x^{\frac{p-1}{5}}, y^{\frac{p-1}{5}}, z^{\frac{p-1}{5}}, w^{\frac{p-1}{5}}).
$$
If $u_1,u_2,u_3$ are distinct and not equal to $1$ in the group of 5-th roots of unity $\mu_5$ then:
$G_2(1,u_1,u_2,u_3)=0$. This gives rise to $24((p-1)/5)^3$ points in $F_2=0$. If $p=31$, there are extra points
in the surface, coming from $u_1=u_2=u_3$ not equal to $1$ in $\mu_5$, totaling $28((p-1)/5)^3$ points. This example, when $p=31$ is given in \cite{Voloch}.



\bigskip
\bibliographystyle{plain}

\end{document}